\documentclass[11pt]{amsart}

\usepackage[english]{babel}
\usepackage{amssymb,amsmath,amsthm}
\usepackage{verbatim}

\usepackage{hyperref}

\usepackage[all]{xy}
\usepackage{xcolor}

\newtheorem{dummy}{dummy}[section]

\newtheorem{theorem}[dummy]{Theorem}
\newtheorem{corollary}[dummy]{Corollary}
\newtheorem{lemma}[dummy]{Lemma}
\newtheorem{proposition}[dummy]{Proposition}
\newtheorem{question}[dummy]{Question}
\newtheorem{conjecture}[dummy]{Conjecture}

\theoremstyle{remark}

\newtheorem{remark}[dummy]{Remark}
\newtheorem{example}[dummy]{Example}

\def\A{\mathbb A}
\def\C{\mathbb C}

\def\L{\mathbb L}
\def\P{\mathbb P}

\def\Z{\mathbb Z}

\def\CC{\mathcal C}

\def\EE{\mathcal E}
\def\FF{\mathcal F}

\def\LL{\mathcal L}

\def\NN{\mathcal N}
\def\OO{\mathcal O}

\newcommand{\BD}{\mathbb{D}}
\newcommand{\KVar}{K_0(\Var/\kk)}
\newcommand{\KDG}{K_0(\dgcat/\kk)}

\newcommand{\cC}{\CC}
\newcommand{\cE}{\EE}
\newcommand{\cF}{\FF}

\newcommand{\cL}{\LL}
\newcommand{\cN}{\NN}
\newcommand{\cO}{\OO}
\newcommand{\cQ}{{\mathcal Q}}
\newcommand{\cU}{{\mathcal U}}
\newcommand{\kk}{\Bbbk}

\DeclareMathOperator{\id}{\mathrm{id}}
\DeclareMathOperator{\Spec}{\mathrm{Spec}}
\DeclareMathOperator{\Ker}{\mathrm{Ker}}
\DeclareMathOperator{\Coker}{\mathrm{Coker}}
\DeclareMathOperator{\Gr}{\mathrm{Gr}}

\DeclareMathOperator{\Bl}{\mathrm{Bl}}
\DeclareMathOperator{\Sing}{\mathrm{Sing}}
\DeclareMathOperator{\CH}{\mathrm{CH}}

\def\={\;=\;}
\def\bal{\begin{aligned}}
\def\eal{\end{aligned}}
\def\be{\begin{equation}\label}
\def\ee{\end{equation}}

\def\wt{\widetilde}
\def\ol{\overline}

\def\mod#1{\; \left({\rm mod} \; #1\right)}

\DeclareMathOperator{\Var}{\mathrm{Var\!}}
\DeclareMathOperator{\dgcat}{\mathrm{DG-cat\!}}

\DeclareMathOperator{\Br}{Br}

\def\Hom {\operatorname{Hom}\nolimits}

\title[D- and L-equivalence]{Grothendieck ring of varieties, D-{} and L-equivalence,\\ and families of quadrics}
\author{Alexander Kuznetsov}
\thanks{A.K.\ was partially supported by the Russian Academic Excellence Project “5-100”, by RFBR grants 15-01-02164 and 15-51-50045, and by the Simons foundation.}
\address{{\sloppy
\parbox{0.95\textwidth}{
{\bf A.K.:} Algebraic Geometry Section, Steklov Mathematical Institute of Russian Academy of Sciences,\\
8 Gubkin str., Moscow 119991 Russia
\\[5pt]
The Poncelet Laboratory, Independent University of Moscow
\hfill\\[5pt]
Laboratory of Algebraic Geometry, National Research University Higher School of Economics
}\bigskip}}
\email{akuznet@mi.ras.ru}
\author{Evgeny Shinder}
\address{{\sloppy
\parbox{0.95\textwidth}{
{\bf E.S.:} School of Mathematics and Statistics,\\
University of Sheffield, S3 7RH, UK.
}\bigskip}}
\email{e.shinder@sheffield.ac.uk}
\date{} % Activate to display a given date or no date (if empty),
\begin{document}

\maketitle

\begin{abstract}
We discuss a conjecture saying that derived equivalence of smooth projective simply connected varieties implies that 
the difference of their classes in the Grothendieck ring of varieties is annihilated by a power of the affine line class.
We support the conjecture with a number of known examples, and one new example. 
We consider a smooth complete intersection $X$ of three quadrics in $\P^5$ and the corresponding double cover $Y \to \P^2$
branched over a sextic curve. We show that as soon as the natural Brauer class on~$Y$ vanishes, so that $X$ and $Y$ are derived equivalent,
the difference $[X] - [Y]$ is annihilated by the affine line class.
\end{abstract}
         
\section{Introduction}

Let $\kk$ be a field and $\KVar$
be the Grothendieck ring of varieties over $\kk$, that is the ring generated by isomorphism classes $[X]$ of algebraic varieties $X$ over $\kk$ with relations
\begin{equation*}
[X] = [Z] + [U] 
\end{equation*}
for every closed subvariety $Z \subset X$ with open complement $U \subset X$.
The product structure is induced by products of varieties, and the unit is $[\Spec(\kk)]$, the class of a point.

The ring $\KVar$ is a very basic object of algebraic geometry introduced by A.\,Grothendieck in his correspondence with J.\,P.\,Serre.
The class of an algebraic variety $X$ in that ring is an important invariant that has a clear ``motivic nature''.
In this aspect, it is a close relative of the class of the derived category~$\BD(X)$ of $X$ in the Bondal--Larsen--Lunts ring $\KDG$ (see~\cite{BLL}),
or of the motive~$M(X)$ of $X$ in the category of Chow motives (or in the Voevodsky category of geometric motives).
It is an interesting and important question to understand relations between these invariants.

\subsection{D-equivalence}

We write $\BD(X)$ for the bounded derived category of coherent sheaves on a smooth projective variety $X$.
We say that smooth projective varieties are {\sf D-equivalent}, if $\BD(X) \cong \BD(Y)$.
The relation between derived categories and motives was discussed by Orlov in~\cite{Orl05}, where besides other things 
there was suggested a conjecture stating D-equivalence of two smooth projective varieties implies an equality of their Chow motives with rational coefficients.
The question that we suggest to discuss is a relation between the other two invariants.

\begin{question}
Assume $X$ and $Y$ are \textup{D}-equivalent smooth projective varieties.
What is the relation between the classes $[X]$ and $[Y]$ in the Grothendieck ring $\KVar$?
\end{question}

There are quite many examples of D-equivalent varieties:
among them there are birational examples, K3 surfaces, Abelian varieties, some non-birational Calabi--Yau varieties.
A general source of derived equivalences is provided by homological projective duality~\cite{Kuz07,Kuz14} for varieties with rectangular Lefschetz decompositions.

For some of the D-equivalences, the classes in the Grothendieck ring of the involved varieties were related (we list some examples below).
In some cases, the relation is rather trivial, but in other cases, it is a bit unexpected.

To explain the relation denote the class of an affine line in the Grothendieck ring by
\begin{equation*}
\L = [\A^1] = [\P^1] - [\Spec(\kk)].
\end{equation*}

Let us start with the most boring situation.
If $X$ or $Y$ has ample canonical or anticanonical class, then a D-equivalence between them implies that $X$ and $Y$ are isomorphic so that $[X] = [Y]$.

It is a bit more interesting to weaken the assumption of ampleness of the canonical class of $X$ to the assumption that $X$ is of general type.
Then if $\BD(X) \cong \BD(Y)$ is a D-equivalence,
the canonical rings of $X$ and $Y$ are isomorphic, hence $X$ and $Y$ are birational, and moreover {\sf K-equivalent} \cite[Theorem~1.4(2)]{Kawamata}
(i.e., there is a smooth projective variety $Z$ with birational morphisms $Z \to X$ and $Z \to Y$ such that the relative canonical classes are equal, i.e., $K_{Z/X} = K_{Z/Y}$).
We discuss K-equivalent varieties in the following example.

\begin{example}\label{example:k-equivalence}
Assume $X$ and $Y$ are K-equivalent.
A conjecture of Kawamata~\cite[Conjecture~1.2]{Kawamata} says that $X$ and $Y$ are then D-equivalent.
On the other hand, motivic integration~\cite{Kontsevich} proves that the classes of $X$ and $Y$ are equal in the completion of the localization~$\KVar[\L^{-1}]$ of the Grothendieck ring
with respect to the dimension filtration of $\KVar[\L^{-1}]$.
The explicit statement one gets unraveling the completion and localization is
that there are two sequences of integers $r_i$, $d_i$ such that the difference~$r_i - d_i$ tends to infinity while the class $([X] - [Y])\L^{r_i}$ 
can be expressed as a linear combination of varieties of dimension at most $d_i$.
If one could avoid the completion at this point, that would just mean that $[X]\L^r = [Y]\L^r$ in the Grothendieck ring for some integer $r$.

Note that in cases when we can actually {\it prove}\/ D-equivalence of K-equivalent varieties $X$ and $Y$, we can even show that $[X] = [Y]$.
For instance, consider the situation of a simple flop (see~\cite[Theorem~3.6]{BO95} for a proof of D-equivalence in that case).
In other words, assume that $X$ and $Y$ are smooth varieties containing ruled subvarieties $\P_S(\cE) \hookrightarrow X$ and $\P_S(\cF) \hookrightarrow Y$ with the same base~$S$ 
and locally free sheaves $\cE$ and $\cF$ on $S$ of the same rank such that
$X \setminus \P_S(\cE) \cong Y \setminus \P_S(\cF)$. 
Of course we have $[X \setminus \P_S(\cE)] = [Y \setminus \P_S(\cF)]$. 
Moreover, $[\P_S(\cE)] = [S][\P^d] = [\P_S(\cF)]$ (where $d + 1$ is the rank of~$\cE$ and~$\cF$).
Summing up these equalities we deduce $[X] = [Y]$.
\end{example}

\begin{example}
In~\cite{Uehara} an example of birational, D-equivalent, but not K-equivalent varieties is constructed.
The example is provided by an appropriate pair of complex rational elliptic surfaces $X$ and~$Y$.
Then $[X] = [Y]$
% $X$ and $Y$ are trivially L-equivalent 
because D-equivalence of complex surfaces implies equality of their Picard numbers,
and for a complex rational smooth projective surface $X$ with Picard number $\rho$ one has $[X] = 1 + \rho \L + \L^2$.
\end{example}

So far, one could imagine that D-equivalent varieties have the \emph{same} class in the Grothendieck ring.
This, however, is far too naive, as Examples~\ref{example:pfaffian-grassmannian} and~\ref{example:g2} show.

\begin{example}\label{example:pfaffian-grassmannian}
Let $X$ and $Y$ be Calabi--Yau threefolds from the Pfaffian--Grassmannian correspondence, see~\cite{BC}.
Then $\BD(X) \cong \BD(Y)$ by~\cite{BC} and~\cite{Kuz06b}.
On the other hand, Borisov showed in~\cite{Bor} that $[X] - [Y]$ is annihilated by a power of $\L$,
% (so $X$ and $Y$ are L-equivalent), 
and Martin \cite{M} improved his result by showing that 
\begin{equation*}
([X] - [Y])\L^6 = 0.
\end{equation*}
It is not known whether $6$ is the minimal power of $\L$ annihilating $[X] -  [Y]$, but it 
is not hard to show that this power is positive, 
i.e., that $[X] \ne [Y]$ in the Grothendieck ring (see Proposition~\ref{proposition:trivial-l-equivalence}).
\end{example}

\begin{example}\label{example:g2}
Let $X$ and $Y$ be Calabi--Yau threefolds of Ito--Miura--Okawa--Ueda, \cite{IMOU}.
By \cite{Kuz16} we have
% One can show that 
$\BD(X) \cong \BD(Y)$. 
On the other hand, the main result of~\cite{IMOU} is the relation 
\begin{equation*}
([X] - [Y])\L = 0.
\end{equation*}
Again, it is easy to show that $[X] \ne [Y]$, so the power of $\L$ annihilating the difference in this case is definitely minimal.
\end{example}

\subsection{L-equivalence}

Examples \ref{example:pfaffian-grassmannian} and~\ref{example:g2} suggest a less naive conjecture.

We say that varieties $X$ and $Y$ are {\sf L-equivalent}, if 
\begin{equation*}
([X] - [Y]) \L^r = 0
\end{equation*}
for some integer $r \ge 0$,
in other words, 
if the classes of $X$ and $Y$ in the localized Grothendieck ring $\KVar[\L^{-1}]$ are equal, i.e., if
% the difference $[X] - [Y]$ is contained in the kernel of the localization homomorphism 
\begin{equation*}
[X] - [Y] \in \Ker(\KVar \to \KVar[\L^{-1}]).
\end{equation*}
If $[X] = [Y]$ in $\KVar$, we say that $X$ and $Y$ are {\sf trivially L-equivalent}.

It is easy to see that L-equivalence and trivial L-equivalence are indeed equivalence relations on the set of isomorphism classes of smooth
projective algebraic varieties over $\kk$.

The following conjecture is quite challenging. 
Its far reaching consequences are discussed below.
%See for instance a discussion of its consequences below.

\begin{conjecture}\label{conjecture:d-l}
If $X$ and $Y$ are smooth projective simply connected varieties such that $\BD(X) \cong \BD(Y)$, then there is a nonnegative integer $r \ge 0$ such that $([X] - [Y])\L^r = 0$.
In other words, \textup{D}-equivalence of varieties implies their \textup{L}-equivalence.
\end{conjecture}

The assumption of simple connectedness in the conjecture is necessary.
As we were informed by~A.~Efimov and independently by~K.~Ueda, one can construct examples of D-equivalent abelian varieties which are not L-equivalent
(the invariant distinguishing between them is the integral Hodge structure on first cohomology group).

Of course, it is possible that the relation between the classes of $X$ and $Y$ is more complicated. 
For instance, it well may be that one has to consider the same completion of the localized Grothendieck ring as in motivic integration, and just say that the classes of $X$ and $Y$ in that completion are the same.

Conjecture \ref{conjecture:d-l} is analogous to Orlov's conjecture \cite[Conjecture 1]{Orl05} saying that Chow motives with rational coefficients of D-equivalent varieties are isomorphic,
and similarly to Orlov's conjecture and to Kawamata's conjecture the converse implication 
definitely does not work in general --- L-equivalence of varieties does not imply their D-equivalence.
The simplest example here is provided by $\P^1 \times \P^1$ and the blowup of $\P^2$ in a point.
Both these surfaces have class $1 + 2\L + \L^2$ in the Grothendieck ring, but are not D-equivalent by \cite{BO}.
Note however that for any pair $X$, $Y$ of L-equivalent varieties their derived categories have the same class
$[\BD(X)] = [\BD(Y)]$ in the Bondal--Larsen--Lunts ring $\KDG$ because the homomorphism $\KVar \to \KDG$ sends
$\L$ to $1$ and thus factors through $\KVar[\L^{-1}]$.

Before going further, let us discuss some consequences of Conjecture~\ref{conjecture:d-l}.

First of all, L-equivalence of varieties $X$ and $Y$ implies equality of their Hodge numbers.
Indeed, evaluating the Hodge polynomial (considered as a map $\mathsf{h} \colon \KVar \to \Z[u,v]$) on the left and right hand sides of the relation $[X]\L^r = [Y]\L^r$, 
we deduce an equality $\mathsf{h}_X(u,v)(uv)^r = \mathsf{h}_Y(u,v)(uv)^r$ in $\Z[u,v]$, which of course implies an equality of Hodge polynomials $\mathsf{h}_X(u,v) = \mathsf{h}_Y(u,v)$ and hence an equality of the Hodge numbers.
A similar argument shows that L-equivalent varieties also have the same (motivic) zeta-functions.
Thus, Conjecture~\ref{conjecture:d-l} predicts equality of Hodge numbers, zeta-functions (and any other multiplicative motivic invariant 
whose value on the class $\L$ is not a zero divisor) for any pair of D-equivalent simply connected varieties.

Besides that, validity of Conjecture~\ref{conjecture:d-l} would provide a lot of nontrivial elements of the Grothendieck ring annihilated by a power of $\L$.
Indeed, the difference of any pair of derived equivalent non-birational Calabi--Yau varieties gives such an element, see Proposition~\ref{proposition:trivial-l-equivalence}.
So, a natural question to ask in this direction is:

\begin{question}\label{question:l-annihilator}
Do differences $[X]-[Y]$ of \textup{L}-equivalent smooth projective varieties generate the 
kernel of the localization homomorphism $\KVar \to \KVar[\L^{-1}]$?
\end{question}

A positive answer to this question would be helpful in certain birational geometry problems, for instance
in rationality questions for cubic hypersurfaces \cite{GS}.

Let us also discuss how Conjecture~\ref{conjecture:d-l} could be proved.
Of course, a natural way to attack it would be by considering the kernel of a Fourier--Mukai functor providing an equivalence of derived categories and cooking a relation in the Grothendieck ring from it.
It is not, however, clear, how this approach could be realized, even in the simplest case when $X$ is a $K3$ surface, $Y$ is a two-dimensional fine moduli space  
of sheaves on $X$ and the Fourier--Mukai kernel is given by the universal sheaf on $X \times Y$.
Note however, that in all examples, including the main result of this paper (see Theorem~\ref{theorem:main} below), 
the proof of L-equivalence of $X$ and $Y$ uses a correspondence that is evidently related to such a Fourier--Mukai kernel.

In this direction one more question seems natural.

\begin{question}
Assume $X$ and $Y$ are both \textup{D}-equivalent and \textup{L}-equivalent.
What is the categorical meaning of the minimal integer $r$ such that $([X] - [Y])\L^r = 0$?
\end{question}

It seems plausible that this integer is related to the rank of a Fourier--Mukai kernel defining an equivalence.
Indeed, in the known cases of D-equivalence in Example~\ref{example:k-equivalence} this kernel can be taken to be of rank~0 (as an object in the derived category of the product $X \times Y$), 
and as we explained, in these cases the corresponding integer $r$ is also zero.

% To end this section let us mention an example of D-equivalence in which nothing is known about L-equivalence of the involved varieties.

% \begin{example}\label{example:elliptic-curves}
% Let $X$ and $Y$ be D-equivalent genus~1 curves.
% If the base field $\kk$ is algebraically closed of characteristic zero, then $X \cong Y$, see~\cite[Corollary~5.46]{Huybrechts}, hence $[X] = [Y]$.
% In general, however, there are examples of D-equivalent genus 1 curves that are not isomorphic, see~\cite{AKW}.
% \end{example}

%\begin{example}
%\footnote{\red{I guess we should remove this example (and the three references).}}
%Let $X$ and $Y$ be Abelian varieties. 
%There are many examples of D-equivalences among them, for instance an abelian variety $X$ is always D-equivalent to its dual $X^\vee$, \cite{MukaiAbelian}.
%A general criterion of D-equivalence was proved by Orlov in~\cite[Theorem~2.19]{OrlovAbelian}.
%One can also consider an extension of this class of examples by considering torsors over Abelian varieties, see~\cite{AKW} for the case of genus~1 curves.
%In all these cases, however, the relation between the classes in the Grothendieck ring of D-equivalent varieties of this type is completely unclear.
%\end{example}

\subsection{$K3$ surfaces and quadrics}

One of the main goals of this paper is to construct a new example of L-equivalence between D-equivalent varieties.
To be more precise, we consider an example of D-equivalent K3 surfaces, in fact, the simplest such example, and prove that these K3 surfaces are L-equivalent.

We start with a K3 surface $X$ of degree $8$ (i.e., an intersection of three quadrics in $\P^5$), 
and take $Y$ to be the corresponding K3 surface of degree 2 (i.e., the double cover of $\P^2$ branched over the sextic discriminant curve of the net of quadrics defining $X$).
Then it is well known (for instance, see~\cite{Kuz08}) that $Y$ carries a natural 2-torsion Brauer class $\alpha_Y \in \Br(Y)$ (this class depends on $X$),
% , to emphasize this we will write sometimes $\alpha_Y(X)$)
such that
\begin{equation*}
\BD(X) \cong \BD(Y,\alpha_Y),
\end{equation*}
the category in the right hand side being the twisted derived category of $Y$ 
(it can be thought of as the derived category of sheaves of coherent modules over the corresponding Azumaya algebra on $Y$).
In general, the Brauer class $\alpha_Y$ is not trivial, and we show that it vanishes if and only if $X$ contains a curve of odd degree.
We prove that as soon as $\alpha_Y = 0$ so that $\BD(X) \cong \BD(Y)$, $X$ and $Y$ are also L-equivalent.
Moreover, we show that in general this L-equivalence is not trivial (it does not follow from an isomorphism of K3 surfaces).

More precisely, our result can be stated as follows.

\begin{theorem}\label{theorem:main}
Let $\kk$ be a field of characteristic not equal to~$2$.
Let $X$ be a smooth intersection of three quadrics in $\P^5$. 
Assume that the corresponding double cover $Y \to \P^2$ is also smooth
and let $\alpha_Y \in \Br(Y)$ be the corresponding Brauer class.
If $X$ has a $\kk$-point 
not lying on a line contained in $X$
and $\alpha_Y = 0$ then $X$ and $Y$ are \textup{D}-equivalent and \textup{L}-equivalent; explicitly
\begin{equation*}
\BD(X) \cong \BD(Y)
\qquad\text{and}\qquad
([X] - [Y])\L = 0.
\end{equation*}
Moreover, if $\kk = \C$ there is a countable union $M'$ of locally closed codimension~$1$ subsets of the moduli space 
of polarized \textup{K3} surfaces of degree $8$ such that for every $X$ in $M'$ we have $\alpha_Y = 0$, but $[X] - [Y] \ne 0$.
\end{theorem}

Existence of a line on $X$ is a codimension one condition on the moduli space of K3 surfaces of degree~$8$.
If $X$ contains a line defined over $\kk$ then %, it also has a $\kk$-point and 
the Brauer class $\alpha_Y$ vanishes automatically.
However, in this case $Y \cong X$, so that the constructed L-equivalence is trivial.

The components of the subset $M'$ mentioned in the theorem are the moduli spaces of K3 surfaces with rank 2 Neron--Severi lattice and some special values of the discriminant, see Lemma~\ref{lemma:non-iso} for details.

Our approach to the theorem is based on studying families of quadrics and relations between their classes in the Grothendieck ring.
% To prove this result we investigate the class in the Grothendieck ring of the total space of the family of quadrics in $\P^5$ passing through $X$.
We establish a general result in this direction by proving for a flat family of $n$-dimensional quadrics $\cQ \to S$ that has a regular section a relation
\begin{equation*}
[\cQ] = [S](1 + \L^n) + [\bar\cQ]\L,
\end{equation*}
where $\bar\cQ$ is the hyperbolic reduction of $\cQ$ --- a family of quadrics over $S$ of dimension $n-2$ constructed from the section.
In case $n = 2$ we prove a strengthening of this result --- first we show that $\bar\cQ$ is nothing but the determinant double cover of $S$ associated to the family $\cQ$;
second we check that one can use a rational section of the family instead of a regular one to get the same relation;
finally we show that existence of a rational section is equivalent to the vanishing of the associated Brauer class.
We discuss this material in Section~\ref{sec:quadrics}.

The simple relations in the Grothendieck ring associated with families of quadrics that we discussed above is probably only the tip of an iceberg.
They can be considered as the Grothendieck ring analogs
of the simplest results about motives of quadrics over non algebraically closed fields, see e.g. \cite[Proposition 2]{Rost2}.
A good question is to find other relations between quadrics in the Grothendieck ring and to understand what are the Grothendieck ring 
incarnations of such fascinating objects as Rost motives \cite{Rost2}.
A related result is the computation by Koll\'ar of the subring of $\KVar$ generated by classes~$[C]$ of conics over $\kk$ \cite{Kol-conics}.

In Section~\ref{section:intersections} we apply these results as follows.
We start with a family of quadrics through $X$ that we think of as a family $Q \to \P^2$ of four-dimensional quadrics over $\P^2$.
We use the hyperbolic reduction with respect to a $\kk$-point $P \in X$ to reduce the dimension of the family of quadrics $Q$ by 2, and get in this way a family $\bar{Q}_P$ of two-dimensional quadrics over $\P^2$.
We show that its total space is isomorphic to the blowup of $\P^4$ with center in $X'$, the projection of $X$ to $\P^4$ from $P$.
Using this blowup representation to control the Chow groups of $\bar{Q}_P$, we show that the family $\bar{Q}_P \to \P^2$ has a rational section if and only if $X$ contains a curve of odd degree.
Since the determinant double cover of $\P^2$ associated with the family of quadrics $\bar{Q}_P$ is nothing but $Y$, we deduce a relation $[\bar{Q}_P] = [\P^2](1 + \L^2) + [Y]\L$,
which in a combination with the blowup relations for $\bar{Q}_P \to \P^4$ and $X' \to X$ implies $[X]\L = [Y]\L$.

% Applying this machinery to the family of quadrics through $X$ gives the required relation.
% 
% To be more precise, we use a $\kk$-point $P \in X$ to reduce the dimension of the family $Q$ of quadrics through $X$ by 2, and get in this way a family $\bar{Q}_P$ of two-dimensional quadrics over $\P^2$.
% We show that its total space is isomorphic to the blowup of $\P^4$ with center in $X'$, the projection of $X$ to $\P^4$ from $P$.
% Using this blowup representation to control the Chow groups of $\bar{Q}_P$, we show that the family $\bar{Q}_P \to \P^2$ has a rational section if and only if $X$ contains a curve of odd degree.

For the second part of the Theorem we use a criterion of Madonna and Nikulin~\cite{MN} for an isomorphism of $X$ and $Y$.
It works very nicely when the Neron--Severi group of $X$ is of rank 2, and in this case it is formulated in terms of arithmetic properties of its discriminant.
We check that there is a countable set of discriminants that ensures that the surfaces are not isomorphic, and kills the Brauer class $\alpha_Y$ at the same time.

Of course, it would be very interesting to consider other examples of D-equivalent K3 surfaces.
In an example of Section~\ref{subsection:verra} we discuss a pair of K3-surfaces $Y_1$ and $Y_2$ of degree 2 associated with a Verra fourfold.
Each of them comes with a natural Brauer class, and we show that as soon as both classes vanish, we have an L-equivalence $([Y_1] - [Y_2])\L = 0$.
However, it is not clear yet that the surfaces are not isomorphic with this assumption.
We plan to discuss this question in future.

Also, in a work in progress we construct an L-equivalence related to a derived equivalence between a K3 surface $X$ of degree 16 and a twisted K3~surface $Y$ of degree~4.
We show that, similarly to the situation in Theorem~\ref{theorem:main}, as soon as the Brauer class on the degree 4 surface $Y$ vanishes, we have an L-equivalence $([X] - [Y])\L^3 = 0$ between $X$ and $Y$.

We also have a result for K3 surfaces of degree 12. 
It is one of classical Mukai's examples that there is an involution on the moduli space of such K3 surfaces such that K3 surfaces in involution are D-equivalent 
(note that a Brauer class does not appear here, so this case, in a sense, is much simpler then the cases discussed above).
We can show, in fact, that $([X] - [Y])\L^7 = 0$ for such K3 surfaces. 
More precise results on L-equivalence in this case has been proved independently by Ito, Miura, Okawa, and Ueda \cite{IMOU-K3}
% (with $\L^3$ factor) 
(with relation $([X] - [Y])\L^3 = 0$)
and by Hassett and Lai \cite{HL} 
(with relation $([X] - [Y])\L = 0$).
% (with $\L$ factor).
%However, we learned that in a joint work of Hashimoto, Ito, Miura, Okawa, and Ueda \cite{HIMOU} a better relation $([X] - [Y])\L^4 = 0$ was proved.
%We do not know, however, whether this relation is minimal.

\medskip

The authors would like to thank Tom Bridgeland, Sasha Efimov, Sergey Galkin, Valery Lunts, Viacheslav Nikulin, Shinnosuke Okawa, Dima Orlov, Alex Perry, Alexander Vishik,
Kazushi Ueda, Ziyu Zhang for discussions and interest in our work, Daniel Huybrechts for his comments on a draft of this paper
and the referee for the suggestions on improving the exposition.
%The authors would like to thank Tom Bridgeland for useful conversations and encouragement, Alexander Vishik for his help with quadrics,
%Sergey Galkin, Valery Lunts, Shinnosuke Okawa, Alex Perry, and Kazushi Ueda for discussions of L-equivalence, Dima Orlov for a discussion of Abelian varieties,
%\blue{V.V.\,Nikulin for discussion of moduli spaces of sheaves on K3 surfaces,}
%and Daniel Huybrechts for his comments on a draft of this paper. 
The authors are grateful to the Higher School of Economics and the School of Mathematics and Statistics at the University of Sheffield
for providing numerous opportunities for collaboration.

\section{Quadric fibrations in the Grothendieck ring}
\label{sec:quadrics}

Unless stated otherwise all schemes are of finite type over a field $\kk$ of characteristic not equal to $2$.

\subsection{Preliminaries on the Grothendieck ring}

We start with discussing a couple of well-known properties of the Grothendieck ring.
The first is elementary.

\begin{lemma}\label{lemma:zltf}
Assume $M \to S$ is a Zariski locally trivial fibration with fiber $F$.
Then $[M] = [S][F]$.
\end{lemma}
\begin{proof}
% \blue{
% If $U \subset S$ is an open subset over which the fibration is trivial then $M_U \cong U \times F$, hence $[M_U] = [U][F]$.
% The statement now follows by applying Noetherian induction to $S$.}
% % When $S$ is empty, there is nothing to prove.
% % On the other hand, $M_S \to S$ is a Zariski locally trivial fibration with fiber $F$, hence by induction $[M_S] = [S][F]$.
% % Summing up these relations, we deduce the lemma.
We use Noetherian induction on $S$.
When $S$ is empty, there is nothing to prove.
Otherwise, let $U \subset S$ be an open subset over which the fibration is trivial.
Then $M_U \cong U \times F$, hence $[M_U] = [U][F]$.
On the other hand, if $Z = S \setminus U$ then $M_{Z} \to Z$ is a Zariski locally trivial fibration with fiber $F$, hence by induction $[M_{Z}] = [{Z}][F]$.
Summing up these relations, we deduce the lemma.
\end{proof}

The next property is much less trivial and relies on deep results in birational geometry.%uses a deep result of Larsen and Lunts.

\begin{proposition}{\cite[Corollary 1]{LS}}\label{proposition:trivial-l-equivalence}
Let $\kk$ be a field of characteristic zero.
Assume $X$ and $Y$ are smooth projective Calabi--Yau varieties over $\kk$.
If $[X] = [Y]$ in $\KVar$ then $X$ is birational to $Y$.
\end{proposition}
\begin{proof}
If $[X] = [Y]$ then of course the images of $[X]$ and $[Y]$ in the quotient ring $\KVar/\L$ are equal.
By~\cite{LL} this quotient is isomorphic to the free abelian group generated by stable birational classes of smooth projective varieties.
So, it follows that $X \times \P^m$ is birational to $Y \times \P^n$ for some integers $m$ and~$n$.
Let $X \times \P^m \dashrightarrow \bar{X}$ and $Y \times \P^n \dashrightarrow \bar{Y}$ be the maximal rationally connected quotients \cite{KMM}.
Then $\bar{X}$ and $\bar{Y}$ are birational, since $X \times \P^m$ and $Y \times \P^n$ are.
Moreover, the map $X \times \P^m \dashrightarrow \bar{X}$ evidently factors through $X$, hence $\bar{X}$ is the maximal rationally connected quotient of $X$, and since $X$ is a Calabi--Yau variety, $\bar{X}$ is birational to $X$.
Analogously, $\bar{Y}$ is birational to $Y$.
Thus $X$ is birational to $Y$.
\end{proof}

\subsection{Families of quadrics and hyperbolic reduction}

Let $S$ be an algebraic variety and $p \colon Q \to S$ a flat family of $n$-dimensional quadrics over $S$.
In other words, we assume that there is a vector bundle~$\cE$ of rank $n + 2$ over $S$, a line bundle $\cL$, and an embedding of vector bundles $q \colon \cL \to S^2\cE^\vee$.
A natural isomorphism
\begin{equation*}
H^0(\P_S(\cE),\pi^*\cL^\vee \otimes \cO_{\P_S(\cE)/S}(2)) \cong
H^0(S,\cL^\vee \otimes S^2\cE^\vee) \cong
\Hom(\cL,S^2\cE^\vee),
\end{equation*}
where $\pi \colon \P_S(\cE) \to S$ is the projection, associates to $q$ a section of the line bundle $\pi^*\cL^\vee \otimes \cO_{\P_S(\cE)/S}(2)$ on~$\P_S(\cE)$.
The family of quadrics associated with $q$ is its zero locus $Q \subset \P_S(\cE)$, so that we have a diagram
\begin{equation*}
\xymatrix{
Q \ar@{^{(}->}[rr] \ar[dr]_p && \P_S(\cE) \ar[dl]^\pi 
\\
& S
}
\end{equation*}
Given a family of quadrics as above we consider the sheaf
\begin{equation*}
\cC := \Coker(\cE \xrightarrow{\ q\ } \cE^\vee \otimes \cL^\vee),
\end{equation*}
called {\sf the cokernel sheaf} of the family.
We denote by $S_{\ge k} \subset S$ the locus where the corank of the map~$q$ is at least $k$.
Equivalently, this is the locus where the rank of the sheaf $\cC$ is at least $k$.
The locus $S_{\ge k} \subset S$ has a natural scheme structure (its ideal is generated by the minors of $q$ of appropriate size).
Note that flatness of the quadric fibration $p: Q \to S$ is equivalent to the rank of the quadrics being everywhere nonzero, that is to $S_{\ge n+2} = \emptyset$.

For every $k \ge 0$ we denote by $F_k(Q/S)$ the relative Hilbert scheme of projective $k$-spaces in fibers of~$Q$ over~$S$.
We denote by $p_k \colon F_k(Q/S) \to S$ the natural projection.
We have a diagram
\begin{equation*}
\xymatrix{
F_k(Q/S) \ar@{^{(}->}[rr] \ar[dr]_{p_k} && \Gr_S(k+1,\cE) \ar[dl]^{\pi_k}
\\
& S
}
\end{equation*}
where $\Gr_S(k+1,\cE)$ is the relative Grassmannian of $(k+1)$-dimensional subspaces in the fibers of $\cE$ and~$\pi_k$ is its natural projection.
Note that $F_0(Q/S) = Q$ and $p_0 = p$.

We define a {\sf $k$-section} of $Q \to S$ to be a regular morphism $s \colon S \to F_k(Q/S)$ such that $p_k \circ s = \id_S$.
Thus a $0$-section of $Q$ is simply a section of $Q \to S$.
We call a $k$-section $s$ {\sf nondegenerate} if for any geometric point $x \in S$ the linear space $s(x) \subset Q_x$ does not intersect $\Sing(Q_x)$.

Giving a $k$-section $s$ is equivalent to giving a vector subbundle $\cU_{k+1} \hookrightarrow \cE$ of rank $k+1$, which is isotropic with respect to $q$, i.e., such that the composition
\begin{equation*}
\cU_{k+1} \hookrightarrow \cE \xrightarrow{\ q\ } \cE^\vee \otimes \cL^\vee \twoheadrightarrow \cU_{k+1}^\vee \otimes \cL^\vee
\end{equation*}
is zero.
A section $s_{\cU_{k+1}}$ corresponding to a vector subbundle $\cU_{k+1}$ is nondegenerate if and only if the composition $\cE \xrightarrow{\ q\ } \cU_{k+1}^\vee \otimes \cL^\vee$ of the two maps above is an epimorphism.
Indeed, this map is not an epimorphism if and only if the dual map $\cU_{k+1} \otimes \cL \to \cE^\vee$ is not a monomorphism at some closed point $x \in S$, that is there is a vector $u \in \cU_{k+1,x}$ such that $q(u,\cE_x) = 0$.
But then $u$ is a singular point of the quadric $Q_x$ that lies on the linear space $s(x) = \P(\cU_{k+1,x})$.

\begin{remark}
If the total space $Q$ of a flat family of quadrics is smooth, then every $k$-section is nondegenerate. 
% \blue{This follows for instance from Lemma 1.3.2 of \cite{ABB} (numbered Lemma 1.8 in the arxiv version of \cite{ABB})
% as existence of a degenerate $k$-section would imply existence of a degenerate section locally.}
% % Indeed, we have $\P_S(\cU_{k+1}) \subset Q$, hence the differential of the map $p \colon Q \to S$ is surjective at any point of $\P_S(\cU_{k+1})$, hence every such point is a smooth point of the corresponding fiber.
Indeed, the differential of the composition $\P_S(\cU_{k+1}) \hookrightarrow Q \xrightarrow{p} S$ is surjective (since the projection map $\P_S(\cU_{k+1}) \to S$ is smooth), 
hence the differential of the map $p \colon Q \to S$ is surjective at any point of~$\P_S(\cU_{k+1}) \subset Q$. 
Therefore, every such point is a smooth point of the corresponding fiber.
\end{remark}

If $Q \to S$ is a family of $n$-dimensional quadrics and $s = s_{\cU_{k+1}}$ is its nondegenerate $k$-section, %of the Hilbert scheme $F_k(Q/S) \to S$, 
we define a new family of quadrics as follows.
We consider a complex
\begin{equation}\label{eq:bar-e}
0 \to \cU_{k+1} \xrightarrow{\ \ \ } \cE \xrightarrow{\ q\ } \cU_{k+1}^\vee \otimes \cL^\vee \to 0
\end{equation}
We define $\bar\cE_s$ to be the middle cohomology sheaf of this complex, that is $\bar\cE_s = \cU_{k+1}^\perp / \cU_{k+1} \subset \cE / \cU_{k+1}$,
where $\cU_{k+1}^\perp := \Ker (\cE \xrightarrow{\ q\ } \cU_{k+1}^\vee  \otimes \cL^\vee)$.
The sheaf $\bar\cE_s$ is locally free since $s$ is nondegenerate.
Since this is the only cohomology sheaf of~\eqref{eq:bar-e}, the dual complex
\begin{equation*}
0 \to \cU_{k+1} \otimes \cL \xrightarrow{\ q\ } \cE^\vee \xrightarrow{\ \ \ } \cU_{k+1}^\vee \to 0
\end{equation*}
also has the only cohomology sheaf, which sits in the middle term, and is isomorphic to $\bar\cE_s^\vee$.
Moreover, we have a natural self-adjoint commutative diagram
\begin{equation}\label{eq:e-ebar}
\vcenter{\xymatrix{
0 \ar[r] &
\cU_{k+1} \ar[r] \ar@{=}[d] &
\cE \ar[r]^-q \ar[d]^q &
\cU_{k+1}^\vee \otimes \cL^\vee \ar[r] \ar@{=}[d] &
0
\\
0 \ar[r] &
\cU_{k+1} \ar[r]^-q &
\cE^\vee \otimes \cL^\vee \ar[r] &
\cU_{k+1}^\vee \otimes \cL^\vee \ar[r] &
0
}}
\end{equation}
By passing to the cohomology sheaves of the complexes, it induces a self-adjoint morphism of vector bundles $\bar{q} \colon \bar\cE_s \to \bar\cE_s^\vee \otimes \cL^\vee$, i.e., a family of quadrics of dimension $n - 2k -2$
\begin{equation*}
\bar{Q}_s \subset \P_S(\bar\cE_s).
\end{equation*}
We call it the {\sf hyperbolic reduction} of $Q$ with respect to section $s$ (\emph{cf}.~\cite[Section~1.3]{ABB}).
Note that~$\bar{Q}_s$ is not necessary flat, but there is a simple criterion for its flatness.

\begin{lemma}\label{lemma:corank-bar-s}
The corank stratification for the quadric fibration $\bar{Q}_s$ coincides with the corank stratification for $Q$.
In particular, if $\dim Q/S = n$ and $s$ is a $k$-section, then $\bar{Q}_s$ is flat iff $S_{\ge n-2k} = \emptyset$.
\end{lemma}
\begin{proof}
The corank stratification of $Q$ is the rank stratification for the sheaf $\cC := \Coker(\cE \xrightarrow{\ q\ } \cE^\vee \otimes \cL^\vee)$.
On the other hand, from the diagram~\eqref{eq:e-ebar} it easily follows that $\cC \cong \bar\cC := \Coker(\bar\cE_s \xrightarrow{\ \bar{q}\ } \bar{\cE}_s^\vee \otimes \cL^\vee)$.
Therefore, this stratification coincides with the rank stratification for $\bar\cC$, i.e., with the corank stratification of $\bar{Q}_s$.

Since the nonflat locus of $\bar{Q}_s$ is the same as the zero locus of $\bar{q}$, i.e., the corank $n-2k$ locus of $\bar{q}$ (note that the rank of $\cE$ is equal to $n + 2$ and the rank of~$\bar\cE_s$ is $(n+2)-2(k+1) = n-2k$),
we conclude that~$\bar{q}$ is flat if and only if the locus $S_{\ge n-2k}$ is empty.
\end{proof}

Homological properties of $Q$ and $\bar{Q}_s$ are closely related.
One can check that the sheaves of even parts of Clifford algebras on $S$ corresponding to the families of quadrics $Q$ and $\bar{Q}_s$ are Morita-equivalent 
(this is proved in~\cite[Theorem~1.8.7]{ABB} under assumption $S_{\ge 2} = \emptyset$, but the fact is more general). 
It follows then from~\cite[Theorem~4.2]{Kuz08} that the interesting parts of the derived categories $\BD(Q)$ and $\BD(\bar{Q}_s)$ are the same.

When $S = \Spec(\kk)$, the Chow motive of $Q$ is up to Tate motives equal to the Tate twist of the Chow motive of $\bar{Q}_s$ \cite[Proposition 2]{Rost2},
and the same result must be true for the motives of $Q$ and $\bar{Q}_s$ in the Voevodsky category of motives over the general base $S$.
%, see \cite{Vial,Bouali} for some results on motives of quadric bundles.
%\footnote{\blue{Yes, $CHM(k) \subset DM(k)$, but for motives over an arbitrary base $S$ one needs Voevodsky's category as Chow motives are not defined over a base I think.}
%\red{I understand now. I guess, in the Voevodsky category there should be a triangle $M(\bar{Q}_s) \otimes \L_S^{k + 1} \to M(Q) \to M(S) \otimes M(\P_S^k) \otimes (1 \oplus \L_S^{n-k})$.
%For $k = 01$ this should follow from the blowup triangles, and for $k > 0$ one can use an induction on $k$.
%But of course, this is not for this paper.}}

Our first goal is to relate the classes of $Q$ and $\bar{Q}_s$ in the Grothendieck ring.

\subsection{The hyperbolic reduction relation}

Let $Q \to S$ be a family of $n$-dimensional quadrics and let $s \colon S \to F_k(Q/S)$ be its non-degenerate $k$-section.
Let $\cU_{k+1} \hookrightarrow \cE$ be the subbundle corresponding to the section~$s$ so that we have $\P_S(\cU_{k+1}) \subset Q$.

\begin{proposition}\label{lemma:q-blowup}
The blow up $f \colon Q' \to Q$ with the center in $\P_S(\cU_{k+1})$ fits into a diagram
\begin{equation*}
\xymatrix@C=3em{
&& Q' \ar[dl]_{f} \ar[dr]^{g} \\
\P_S(\cU_{k+1}) \ar@{^{(}->}[r] & Q && \P_S(\cE/\cU_{k+1}) & \bar{Q}_s \ar@{_{(}->}[l] \\
}
\end{equation*}
where the map $g$ has fibers $\P^{k+1}$ over $\bar{Q}_s$ and fibers $\P^k$ over its complement,
and is Zariski locally trivial over both strata.
\end{proposition}
\begin{proof}
The blowup of a projective space $\P^{n+1}$ in a linear subspace $\P^k$ is a $\P^{k+1}$-bundle over a smaller projective space $\P^{n-k}$.
We think of $\P_S(\cU_{k+1}) \subset \P_S(\cE)$ as of a relative version of $\P^k \subset \P^{n+1}$;
then the relative version of $\P^{n-k}$ is $\P_S(\cE/\cU_{k+1})$.
Denote by $H$ and $H'$ the relative hyperplane classes of $\P_S(\cE)$ and~$\P_S(\cE/\cU_{k+1})$ respectively.
Define a vector bundle $\cU_{k+2}$ of rank $k + 2$ on $\P_S(\cE/\cU_{k+1})$ from the diagram
\begin{equation}\label{eq:diagram}
\vcenter{\xymatrix{
0 \ar[r] & {\pi'}^*\cU_{k+1} \ar[r] \ar@{=}[d] & \cU_{k+2} \ar[r] \ar[d] & \OO(-H') \ar[r] \ar[d] & 0
\\
0 \ar[r] & {\pi'}^*\cU_{k+1} \ar[r] & {\pi'}^*\cE \ar[r] & {\pi'}^*(\cE/\cU_{k+1}) \ar[r] & 0,
}}
\end{equation}
where $\pi' \colon \P_S(\cE/\cU_{k+1}) \to S$ is the projection and the top line is the pullback of the bottom line
with respect to the right vertical map.
The relative version of the isomorphism mentioned above is then the isomorphism at the top of the following diagram
\begin{equation*}
\xymatrix{
& Q' \ar[dl]_f \ar@{..>}[drrr]^g \ar@{^{(}->}[r] & \Bl_{\P_S(\cU_{k+1})}(\P_S(\cE)) \ar[dl]_{\tilde{f}} \ar@{=}[r]^\sim & \P_{\P_S(\cE/\cU_{k+1})}(\cU_{k+2}) \ar[dr]^{\tilde{g}}
\\
Q \ar@{^{(}->}[r] & \P_S(\cE) &&& \P_S(\cE/\cU_{k+1})
}
\end{equation*}
We obtain a chain of maps (the dotted arrow of the above diagram)
\begin{equation*}
Q' := \Bl_{\P_S(\cU_{k+1})}(Q) \hookrightarrow \Bl_{\P_S(\cU_{k+1})}(\P_S(\cE)) \cong \P_{\P_S(\cE/\cU_{k+1})}(\cU_{k+2}) \to \P_S(\cE/\cU_{k+1}).
\end{equation*}
Our goal is to describe the fibers of this composition $g \colon Q' \to \P_S(\cE/\cU_{k+1})$.
% To simplify the notation 
To unburden formulas we will use the same notation for the pullbacks of vector bundles with respect to natural morphisms as for the bundles themselves.
% we will omit $\pi'^*$ when considering pullbacks of $\cU_{k+1}$ and $\cL$ to $\P_{\P_S(\cE/\cU_{k+1})}(\cU_{k+2})$.
We will also denote by $D$ the divisor class of the line bundle $\cL^\vee$, and by $E$ the exceptional divisor of the blowup $\Bl_{\P_S(\cU_{k+1})}(\P_S(\cE))$.

We have a natural relation in the Picard group of the blowup
\begin{equation}\label{eq:picard-relation}
H' = H - E.
\end{equation}
On the other hand, $Q' \subset \Bl_{\P_S(\cU_{k+1})}(\P_S(\cE))$ is a divisor in the class $2H + D - E$.
By~\eqref{eq:picard-relation} it can be rewritten as \hbox{$H + H' + D$}.
Since $H$ is also a relative hyperplane class for $\tilde{g} \colon \P_{\P_S(\cE/\cU_{k+1})}(\cU_{k+2}) \to \P_S(\cE/\cU_{k+1})$, while~$D$ and $H'$ are pullbacks from $\P_S(\cE/\cU_{k+1})$, it follows that 
the general fiber of $g \colon  Q' \to \P_S(\cE/\cU_{k+1})$ is a hyperplane $\P^k$ in $\P^{k+1}$, while special fibers are the whole spaces~$\P^{k+1}$.

The locus over which the fibers jump is the zero locus of the section $q'$ of the pushforward of the line bundle $\cO(H + H' + D)$ with respect to the map $\tilde{g}$.
By projection formula, we have
\begin{equation*}
\tilde{g}_*\cO(H + H' + D) \cong 
(\tilde{g}_*\cO(H)) \otimes \cO(H' + D) \cong 
\cU_{k+2}^\vee(H' + D) \cong \cU_{k+2}^\vee \otimes \cL^\vee(H'),
\end{equation*}
and we have a section $q' \in H^0(\P_S(\cE/\cU_{k+1}), \cU_{k+2}^\vee \otimes \cL^\vee(H'))$ corresponding to $Q'$.

To analyze the zero locus of $q'$ we consider two short exact sequences on $\P_S(\cE/\cU_{k+1})$
\begin{equation}\label{eq:seq1}
0 \to \cU_{k+2}^\vee \otimes \cL^\vee(H') \to S^2\cU_{k+2}^\vee \otimes \cL^\vee \to  S^2\cU_{k+1}^\vee \otimes \cL^\vee \to 0,
\end{equation}
\begin{equation}\label{eq:seq2}
0 \to \cL^\vee(2H') \to \cU_{k+2}^\vee \otimes \cL^\vee(H') \to \cU_{k+1}^\vee \otimes \cL^\vee(H') \to 0
\end{equation}
obtained from the defining sequence~\eqref{eq:diagram} of $\cU_{k+2}$ by taking symmetric square, dualizing, and twisting.

The section of the middle term of (\ref{eq:seq1}) given by $q$ vanishes when projected to the third term (since~$\cU_{k+1}$ is an isotropic subbundle of $\cE$), 
and gives a section $q'$ of the first term corresponding to $Q'$.

The projection of $q'$ to the third term of (\ref{eq:seq2}) corresponds under the isomorphism 
\begin{equation*}
H^0(\P_S(\cE/\cU_{k+1}), \cU_{k+1}^\vee \otimes \cL^\vee(H')) \simeq H^0(S, \cU_{k+1}^\vee \otimes \cL^\vee \otimes (\cE/\cU_{k+1})^\vee)
\end{equation*}
to the map $\cE/\cU_{k+1} \xrightarrow{\ q\ } \cU_{k+1}^\vee \otimes \cL^\vee$ of bundles on $S$,
which is surjective by non-degeneracy of $s$, and its kernel is the bundle $\bar{\cE}_s = \cU_{k+1}^\perp / \cU_{k+1}$. 
Thus the zero locus of the projection of~$q'$ to the third term of~\eqref{eq:seq2} is equal to $\P_S(\bar{\cE}_s) \subset \P_S(\cE/\cU_{k+1})$.
When restricted to this projective subbundle, the section $q'$ comes from a section $\bar{q}$ of the first term of (\ref{eq:seq2}).
This section~$\bar{q}$ defines a quadric in $\P_S(\bar{\cE}_s)$, and it is easy to see that it coincides with the quadric $\bar{Q}_s$.

Zariski locally triviality of the maps $g^{-1}(\bar{Q}_s) \to \bar{Q}_s$ and $g^{-1}(\P_S(\cE/\cU_{k+1}) \setminus \bar{Q}_s) \to \P_S(\cE/\cU_{k+1}) \setminus \bar{Q}_s$ 
follows easily from the fact that the divisor $H$ provides a relative hyperplane class for both of them.
\end{proof}

\begin{remark}
When $k = 0$ one can show that the map $g \colon Q' \to \P_S(\cE/\cU_{k+1})$ is the blowup of $\P_S(\cE/\cU_{k+1})$ with center in $\bar{Q}_s$.
We omit the proof of this fact, because we will not need it.
\end{remark}

The geometric relation between $Q$ and $\bar{Q}_s$ described in the above Proposition has an immediate consequence for their Grothendieck ring classes.

\begin{corollary}\label{proposition:hyperbolic-reduction}
Assume $p \colon Q \to S$ is a flat family of $n$-dimensional quadrics that admits a non-degenerate $k$-section $s \colon S \to F_k(Q/S)$.
Then
\begin{equation*}
[Q] = [S][\P^k](1 + \L^{n-k}) + [\bar{Q}_s]\L^{k+1}.
\end{equation*}
\end{corollary}
\begin{proof}
% [Proof of Proposition~\textup{\ref{proposition:hyperbolic-reduction}}]
The map $f \colon Q' \to Q$ is a blowup with smooth center $\P_S(\cU_{k+1})$ of codimension~$n - k$, hence
\begin{equation*}
[Q'] = \Big([Q] - [\P_S(\cU_{k+1})]\Big) + [\P_S(\cU_{k+1})][\P^{n-k-1}] = [Q] + [S][\P^k](\L + \dots + \L^{n-k-1}).
\end{equation*}
Similarly, the map $g$ is a stratified Zariski locally trivial fibration, hence
\begin{equation*}
[Q'] = \Big([\P_S(\cE/\cU_{k+1})] - [\bar{Q}_s]\Big)[\P^k] + [\bar{Q}_s][\P^{k+1}] = [S][\P^{n-k}][\P^k] + [\bar{Q}_s]\L^{k+1}.
\end{equation*}
The two equalities combined imply the required relation.
\end{proof}

\begin{example}
Let us apply Corollary~\ref{proposition:hyperbolic-reduction}  to the case
when $S = \Spec(\kk)$ and $Q$ is a smooth $n$-dimensional quadric over $\kk$. 
Then as soon as $Q$ contains a linear $k$-dimensional subspace defined over~$\kk$ with $k < \lfloor n/2 \rfloor$, one has $[Q] = 1 + \L + \dots + \L^k + [\bar{Q}] \L^{k+1}  + \L^{n-k} + \dots + \L^n$
where $\ol{Q}$ is a $(n-2k-2)$-dimensional quadric.

Furthermore, if $Q$ has a maximal isotropic subspace (i.e., with $k = \lfloor n/2 \rfloor$) defined over $\kk$, then $\ol{Q}$ is empty and Proposition~\ref{lemma:q-blowup} implies that
% \blue{If $Q$ has a maximal isotropic subspace defined over $\kk$, then Corollary~\ref{proposition:hyperbolic-reduction} applies with $[Q_s] = 0$ and yields}
$[Q] = [\P^n] + \L^{n/2}$ (with the second summand present only for even $n$).

Conversely, one may ask whether $[Q] \equiv [\P^n] + \L^{n/2} \mod{\L^{k+1}}$ for some $k \le \left[\frac{n}{2}\right]$ implies existence of a $k$-dimensional linear space on $Q$.
For instance, when the base field $\kk$ has characteristic zero this holds for $k = 0$ 
by the Larsen--Lunts theorem~\cite{LL}: if $[Q] \equiv 1 \mod{\L}$, then $Q$ is stably rational, and hence has a rational point.
\end{example}

In course of proof of Proposition~\ref{lemma:q-blowup} we obtained the following explicit description of the hyperbolic reduction $\bar{Q}_s$ which we will need later:

\begin{lemma}\label{lemma:bar-q-explicit-general}
The subscheme $\bar{Q}_s \subset \P_S(\cE/\cU_{k+1})$ is the zero locus of the section $q'$ of the middle term $\cU_{k+2}^\vee \otimes \cL^\vee(H')$ of the short exact sequence \eqref{eq:seq2}
of vector bundles on $\P_S(\cE/\cU_{k+1})$. 
\end{lemma}

The isomorphism class of the hyperbolic reduction $\bar{Q}_s$ depends in general on the choice of a $k$-section~$s$ (in fact, even the isomorphism class of the vector bundle $\bar\cE_s$ depends on $s$).
However, it easily follows from Corollary~\ref{proposition:hyperbolic-reduction} that the L-equivalence class of $\bar{Q}_s$ is independent of $s$.
The next lemma shows that this L-equivalence in fact is trivial.

% Let us note that even though the family of quadrics $\bar{Q}_s$ may depend on the choice of the section, its class is an invariant of the original family:

\begin{lemma}
The class $[\bar{Q}_s]$ in the Grothendieck ring is independent on the choice of a nondegenerate $k$-section $s$.
\end{lemma}
\begin{proof}
Assume $s$ and $s'$ are two nondegenerate $k$-sections of $Q$.
We want to show that $[\bar{Q}_s] = [\bar{Q}_{s'}]$.
The proof uses Noetherian induction on $S$. 

We may assume that $S$ is an irreducible variety, and let $K = \kk(S)$ be the function field of $S$.
Over $K$ the quadratic form $q_K$ decomposes into an orthogonal direct sum of its hyperbolic reduction and a split (hyperbolic) quadratic form $q_0$ of rank $2k$.
Thus, we have orthogonal direct sum decompositions
\begin{equation*}
q_{K} \simeq \bar{q}_{s,K} \perp q_0
\qquad\text{and}\qquad
q_{K} \simeq \bar{q}_{s',K} \perp q_0.
\end{equation*}
By the Witt Cancellation Theorem for quadratic forms over fields, we have an isomorphis $\bar{q}_{s,K} \cong \bar{q}_{s',K}$.
It extends to an isomorphism of the corresponding quadrics
% is independent on the choice of the section.
% Thus for any two nondegenerate $k$-sections we get an isomorphism 
\begin{equation*}
\bar{Q}_{s,U} \simeq \bar{Q}_{s',U}
\end{equation*}
over some dense open subset $U \subset S$. 
Therefore $[\bar{Q}_{s,U}] = [\bar{Q}_{s',U}]$.
By Noetherian induction we also have $[\bar{Q}_{s,S \setminus U}] = [\bar{Q}_{s',S \setminus U}]$.
Summing up these equalities gives the result.
\end{proof}

\subsection{Reduction to the double cover}

Assume that $n = 2k + 2$ is even.
Recall that the rank of $\cE$ is $n + 2 = 2k + 4$.
Consider the determinant of the map $q \colon \cE \to \cE^\vee \otimes \cL^\vee$.
It is a map 
\begin{equation*}
\det(q) \colon \det(\cE) \to \det(\cE)^\vee \otimes (\cL^\vee)^{2k+4}.
\end{equation*}
In particular, it gives a section of the line bundle $(\det(\cE)^\vee \otimes (\cL^\vee)^{k+2})^2$, whose zero locus coincides scheme-theoretically with the discriminant divisor $S_{\ge 1} \subset S$.
We denote by
\begin{equation}\label{eq:tilde-s}
\tilde{S} := \Spec_S(\cO_S \oplus (\det(\cE) \otimes \cL^{k+2}))
\end{equation}
the corresponding double cover branched along $S_{\ge 1}$, and call it the {\sf determinant double cover}.

A more invariant way to construct $\tilde{S}$ is as follows.
Consider the sheaf of even parts of Clifford algebras on $S$ associated with the family of quadrics $Q$ (see~\cite[Section~3.3]{Kuz08}).
Then the center of this sheaf (in case of even dimension) can be identified as an $\cO_S$-module with $\cO_S \oplus (\det(\cE) \otimes \cL^{k+2})$ and 
its algebra structure, given by the Clifford multiplication, is the one discussed above, see~\cite[Section~3.5]{Kuz08}.

\begin{remark}
Note that the double cover $\tilde{S}$ is well defined for any flat family~$Q$ of even-dimensional quadrics, even in case when $S_{\ge 1} = S$ (i.e., when all quadrics in the family are degenerate).
Indeed, in this case, the sheaf of algebras in the right hand side of~\eqref{eq:tilde-s} is nilpotent, so $\tilde{S}$ in nonreduced, but still it is a flat $S$-scheme.
Note also that formation of the determinant double cover commutes with arbitrary base changes.
\end{remark}

\begin{lemma}\label{lemma:qbar-tildes}
Let $p \colon Q \to S$ be a flat family of quadrics of dimension $n = 2k+2$ with $S_{\ge 2} = \emptyset$.
If $s \colon S \to F_k(Q/S)$ is a nondegenerate regular $k$-section then $\bar{Q}_s \cong \tilde{S}$.
In particular, $\bar{Q}_s$ does not depend on a choice of $s$.
\end{lemma}
\begin{proof}
The family of quadrics $\bar{Q}_s$ is flat (Lemma~\ref{lemma:corank-bar-s}) of dimension 0 over $S$, 
hence the sheaf of even parts of Clifford algebras of $\bar{Q}_s$ is commutative of rank 2 and its (relative over $S$) spectrum is isomorphic to~$\bar{Q}_s$.
On the other hand, this sheaf is Morita-equivalent to the sheaf of even parts of Clifford algebras of $Q$ (see~\cite[Theorem~1.8.7]{ABB}), hence is isomorphic to the center of that algebra.
But the relative spectrum of the center is equal to the determinant double cover $\tilde{S}$ by definition.
\end{proof}

\begin{corollary}\label{corollary:q2-regular-section}
Let $p \colon Q \to S$ be a flat family of quadrics of dimension $n = 2k+2$ with $S_{\ge 2} = \emptyset$.
% a flat family of two-dimensional quadrics whose corank is everywhere less than $2$.
If $p \colon Q \to S$ admits a nondegenerate regular $k$-section then
\begin{equation*}
[Q] = [S][\P^k](1 + \L^{k+2}) + [\tilde{S}]\L^{k+1}.
\end{equation*}
\end{corollary}

\subsection{The Brauer class and rational sections}

In this section we assume $n = 2$, so $p \colon Q \to S$ is a family of two-dimensional quadrics.
Assume furthermore that $S_{\ge 2} = \emptyset$.

\begin{lemma}\label{lemma:fq-p1}
Let $p \colon Q \to S$ be a flat family of two-dimensional quadrics with $S_{\ge 2} = \emptyset$.
The map $p_1 \colon F_1(Q/S) \to S$ factors as
\begin{equation*}
F_1(Q/S) \xrightarrow{\ \tilde{p}_1\ } \tilde{S} \xrightarrow{\ \ \ } S,
\end{equation*}
where $\tilde{S}$ is the determinant double cover of $S$, and the map $\tilde{p}_1 \colon F_1(Q/S) \to \tilde{S}$ is an \'etale locally trivial $\P^1$-bundle.
\end{lemma}
\begin{proof}
By~\cite[Proposition~3.13]{Kuz08} there is a sheaf of Azumaya algebras of degree 2 on~$\tilde{S}$ whose pushforward to $S$ 
is isomorphic to the sheaf of even parts of Clifford algebras of $Q$, and~\cite[Lemma~4.2]{Kuz10} identifies the corresponding $\P^1$-fibration over $\tilde{S}$ with the relative Hilbert scheme $F_1(Q/S)$.
\end{proof}

In what follows we denote by $\alpha_Q \in \Br(\tilde{S})$ the class in the Brauer group
(considered as the group of Morita-equivalence classes of Azumaya algebras)
of the $\P^1$-fibration $F_1(Q/S) \to \tilde{S}$.
By construction, it is 2-torsion.

By a (non-degenerate) rational (multi)section of a morphism we understand a regular (multi)section defined on an open dense subset of the base scheme with the same notion of non-degeneracy as before.

\begin{proposition}\label{proposition:section-alpha}
Let $p \colon Q \to S$ be a flat family of two-dimensional quadrics with $S_{\ge 2} = \emptyset$.
The following conditions are equivalent
\begin{enumerate}
\item\label{item:section} 
there is a rational non-degenerate section $S \dashrightarrow Q$ of $p \colon Q \to S$; 
\item[($1'$)]\label{item:multisection} 
there is a rational non-degenerate multisection of $p \colon Q \to S$ of odd degree; 
\item\label{item:section-f} 
there is a rational section $\tilde{S} \dashrightarrow F_1(Q/S)$ of $\tilde{p}_1 \colon F_1(Q/S) \to \tilde{S}$; 
\item[($2'$)]\label{item:multisection-f} 
there is a rational multisection of $\tilde{p}_1 \colon F_1(Q/S) \to \tilde{S}$ of odd degree; 
\item\label{item:alpha}
$\alpha_Q = 0$, i.e., $\tilde{p}_1 \colon  F_1(Q/S) \to \tilde{S}$ is a projectivization of a rank $2$ vector bundle on $\tilde{S}$.
\end{enumerate}
\end{proposition}
\begin{proof}
Statements $(2)$, $(2')$ and $(3)$ are equivalent because $F_1(Q/S) \to \tilde{S}$ is the relative Severi--Brauer variety corresponding to the degree two Azumaya algebra given by $\alpha_Q$.

Let us prove that $(1')$ is equivalent to $(2')$.

Assume $Z \subset Q$ is a locally closed subscheme such that $Z$ does not intersect singular loci of the fibers of $p$ and the map $Z \to S$ is finite of odd degree $d$ onto an open subset $U \subset S$ (this is a rational non-degenerate multisection).
Let $Z_1 \subset F_1(Q/S)$ be the subscheme parameterizing lines that intersect $Z$.
Clearly, the map $Z_1 \to \tilde{S}$ is finite of degree $d$ on $\tilde{U} \subset \tilde{S}$, hence is a rational multisection of odd degree.

Conversely, assume $Z_1 \subset F_1(Q/S)$ is a locally closed subscheme such that the map $Z_1 \to \tilde{S}$ is finite of odd degree~$d$ onto $\tilde{U} \subset \tilde{S}$ for an open subset $U \subset S$.
If $S_{\ge 1} \ne S$, we may (shrinking $U$ if necessary) assume that $U \cap S_{\ge 1} = \emptyset$. 
Then consider the subscheme $Z \subset Q$ obtained by the intersections of lines parameterized by $Z_1$.
Clearly, the map $Z \to S$ is finite of degree $d^2$ onto $U \subset S$, hence a rational multisection of odd degree, which is evidently non-degenerate.

Now assume $S_{\ge 1} = S$, so all quadrics in the family are cones.
The bases of those cones form a non-degenerate conic bundle $Q' \to S$ and $F_1(Q/S) \cong Q'\times_S \tilde{S}$ (recall that in this case $\tilde{S}$ is a nilpotent thickening of $S$, in particular we have a closed embedding $S \to \tilde{S}$).
Restricting $Z_1 \subset F_1(Q/S)$ to $S \subset \tilde{S}$, we obtain a subscheme $Z'_1 \subset Q'$.
The open complement of cones vertices $Q_0 \subset Q$ comes with a map $Q_0 \to Q'$ which is
a Zariski locally trivial $\A^1$-fibration, hence there is a rational section $Q' \dashrightarrow Q_0$.
If $Z \subset Q_0 \subset Q$ is the image of $Z'_1$ under it, it is a rational nondegenerate multisection of $Q \to S$ of degree~$d$.

Finally, note that the above proof of equivalence of $(1')$ and $(2')$ proves at the same time equivalence of $(1)$ and $(2)$
(just keep $d = 1$ everywhere).
\end{proof}

\begin{corollary}\label{corollary:sections-basechange}
Let $p \colon Q \to S$ be a flat family of two-dimensional quadrics with $S_{\ge 2} = \emptyset$.
If $p$ has a nondegenerate rational section, then for any base change $S' \to S$ the family of quadrics $Q' := Q\times_S S' \to S'$ also has a nondegenerate rational section.
\end{corollary}
\begin{proof}
If $Q \to S$ has a section, then $F_1(Q/S) \cong \P_{\tilde{S}}(\cF)$, where~$\cF$ is a vector bundle of rank 2 on $\tilde{S}$.
It follows that $F_1(Q'/S') \cong \P_{\tilde{S}'}(\cF\vert_{\tilde{S}'})$, hence $Q' \to S'$ also has a nondegenerate rational section.
\end{proof}

\begin{remark}\label{remark:isotropic-quadrics}
Let $p \colon Q \to S$ be a family of quadrics of arbitrary dimension where $S$ is a scheme over the field $\kk$ of characteristic zero.
It is known that existence of a rational section implies existence of a section locally in Zariski topology on $S \setminus S_{\ge 1}$, that
is over the complement of the discriminant~\cite{Panin}. 
Thus, Corollary~\ref{corollary:sections-basechange} is a stronger version of Panin's theorem for $n = 2$.
\end{remark}

\begin{theorem}\label{thm:reduction-dim2}
Let $p \colon Q \to S$ be a flat family of two-dimensional quadrics with $S_{\ge 2} = \emptyset$.
If either of equivalent conditions of Proposition~\textup{\ref{proposition:section-alpha}} holds, then
\begin{equation*}
[Q] = [S](1 + \L^2) + [\tilde{S}]\L.
\end{equation*}
\end{theorem}
\begin{proof}
We use Noetherian induction on $S$. 
When $S$ is empty, there is nothing to prove.
Otherwise, we choose a rational non-degenerate section $s$ of $Q \to S$ and let $U \subset S$ be the open subset where $s$ is regular and non-degenerate.
Applying Corollary~\ref{corollary:q2-regular-section} to $Q_U \to U$ we deduce
\begin{equation*}
[Q_U] = [U](1 + \L^2) + [\tilde{U}]\L.
\end{equation*}
On the other hand, denoting by $S' = S \setminus U$ the complement, by Corollary~\ref{corollary:sections-basechange} the conditions of the Theorem hold true for the restricted family of quadrics $Q' := Q \times_S S' \to S'$.
By induction
\begin{equation*}
[Q'] = [S'](1 + \L^2) + [\tilde{S'}]\L.
\end{equation*}
Summing up these equalities, we obtain the Theorem.
\end{proof}

\subsection{Some applications}

In this section we gather a couple of immediate applications of the above results.

\subsubsection{Cubic fourfolds with a plane}

Let $X \subset \P^5$ be a smooth cubic fourfold containing a $2$-plane $P \subset X$ and let $\wt{X} = \Bl_P(X)$ be the blow up of $P$.
Projecting from $P$ we obtain a flat family of two-dimensional quadrics $p \colon \wt{X} \to \P^2$ with determinant double cover K3 surface $Y$.
By Theorem \ref{thm:reduction-dim2} we get that as soon as $p$ admits a rational section (by \cite[Prop. 4.7]{Kuz10} this is equivalent to existence of an algebraic cycle $T \in \CH^2(X)$ 
such that the intersection degree $T \cdot (H^2 - P)$ is odd), then $[\wt{X}] = [\P^2](1+\L^2) + [Y]\L$ so that
\begin{equation}\label{eq:cubic}
[X] = 1 + \L^2 + \L^4 + [Y]\L.
\end{equation}
This generalizes the computation for the case of $X$ admitting two disjoint planes \cite[Example 5.9]{GS}.

A cubic fourfold as above is rational, and it is plausible
that a smooth cubic fourfold is rational if and only if it has a class in the Grothendieck ring given by (\ref{eq:cubic}) for some K3 surface $Y$.

\subsubsection{Verra fourfolds}
\label{subsection:verra}

Let $X \to \P^2 \times \P^2$ be a double covering branched over a divisor $D \subset \P^2 \times \P^2$ of bidegree $(2,2)$.
The fibers of the projection $p_1 \colon X \to \P^2$ to the first factor are double coverings of~$\P^2$ branched over a conic, hence are two-dimensional quadrics.
Thus $p_1 \colon X \to \P^2$ is a family of two-dimensional quadrics.
Consequently, its determinant double cover $Y_1 \to \P^2$ is a K3 surface that comes with a Brauer class~$\alpha_1$.
Assuming that the Brauer class vanishes, we deduce from Theorem~\ref{thm:reduction-dim2} a relation
% Again, assuming that $p_1$ has a rational section (this is equivalent to existence of a rational map $\P^2 \dashrightarrow \P^2$ such that its graph intersects $D$ along a double curve), we obtain a relation
\begin{equation*}
[X] = [\P^2](1 + \L^2) + [Y_1]\L.
\end{equation*}
% where $Y_1 \to \P^2$ is the determinant double cover of the family $p_1$.
On the other hand, applying the same argument to the projection $p_2 \colon X \to \P^2$ to the second factor, we construct yet another K3 surface double cover $Y_2 \to \P^2$ with a Brauer class $\alpha_2$,
and assuming again that the Brauer class vanishes we obtain
% Assuming again existence of a rational section of $p_2$, we deduce a relation
\begin{equation*}
[X] = [\P^2](1 + \L^2) + [Y_2]\L.
\end{equation*}
% where $Y_2 \to \P^2$ is the determinant double cover of the family $p_2$.
Comparing the two relations we deduce an L-equivalence:
\begin{equation*}
([Y_1] - [Y_2])\L = 0.
\end{equation*}
As we explained above, this holds as soon as the two Brauer classes $\alpha_1 \in \Br(Y_1)$ and $\alpha_2 \in \Br(Y_2)$ vanish.
However, we do not know a reasonable geometric reformulation of this condition, and did not check that there is an example when this condition holds, but the surfaces $Y_1$ and $Y_2$ are not isomorphic.
We plan to return to this question in future.

\section{Complete intersections of quadrics}
\label{section:intersections}

\subsection{Complete intersections of quadrics and the corresponding family of quadrics}

% We apply the results of the previous section in the following situation.
Let $V$ and~$W$ be vector spaces of dimensions 
\begin{equation*}
\dim V = n + 2,
\qquad 
\dim W = m + 1,
\end{equation*}
and choose an embedding $q \colon W \hookrightarrow S^2V^\vee$.
Denote by $Q_w \subset \P(V)$ the quadric with equation $q(w) \in S^2V^\vee$.
Assume that
\begin{equation*}
X := \bigcap_{w \in W} Q_w \subset \P(V)
\end{equation*}
is a complete intersection, i.e., $\dim X = n - m$.

We consider $\{Q_w\}$ as a family of quadrics in $\P(V)$ parameterized by $\P(W)$.
Its total space $Q$ is a divisor in $\P(V) \times \P(W)$ of bidegree~$(2,1)$ with equation $q \in S^2V^\vee \otimes W^\vee$.
So, by definition 
\begin{equation*}
Q \to \P(W), 
\end{equation*}
is a flat family of $n$-dimensional quadrics corresponding to $\cL = \cO_{\P(W)}(-1)$ and $\cE = V \otimes \cO_{\P(W)}$ in the notation of Section \ref{sec:quadrics}.
We denote by $D \subset \P(W)$ its discriminant divisor, and when $n$ is even we denote~by
\begin{equation*}
Y \to \P(W)
\end{equation*}
its determinant double cover.
It is a double cover branched over $D$.

The following criterion of smoothness is well known.

\begin{lemma}\label{lemma:smoothness-criterion}
%Assume $\kk$ is algebraically closed.
A complete intersection of quadrics $X$ is smooth if and only if the family of quadrics is regular, i.e., if $X$ does not intersect the singular locus of any quadric in the family. 
Furthermore, the discriminant divisor $D$ \textup{(}and for even $n$ the determinant double cover $Y$ as well\textup{)} is smooth if and only if the family of quadrics is regular and does not contain quadrics of corank~$2$.
\end{lemma}
\begin{proof}
Assume a closed point $v$ of $X$ lies on the singular locus of a quadric $Q_w$.
The tangent space at $v$ to~$Q_w$ is then equal to the tangent space of~$\P(V)$, hence has dimension $n + 1$.
The tangent space of $X$ is obtained by intersecting it with tangent spaces of $m$ other quadrics, hence it has dimension at least $(n+1) - m = n - m + 1$.
Since the dimension of $X$ is $n-m$, this point is a singularity of $X$.

Conversely, assume $v$ is a singular point of $X$. 
Then the dimension of the tangent space at $v$ to $X$ is at least $n-m+1$, and its codimension in the tangent space of $\P(V) = \P^{n+1}$ is at most $m$.
But this space is the intersection of $(m+1)$ hyperplanes (tangents spaces to the quadrics), hence a linear combination of these hyperplanes is zero.
Then the corresponding linear combination of quadrics is singular at $v$.

For the second part note that $D$ is the intersection of $\P(W) \subset \P(S^2V^\vee)$ with the discriminant divisor $\Delta \subset \P(S^2V^\vee)$.
It is singular at a point $w \in D$ either if $w$ is a singular point of $\Delta$, i.e., the quadric $Q_w$ has corank $\ge 2$, or if $\P(W)$ is tangent to $\Delta$ at $w$.
But the tangent space to $\Delta$ at quadric $q_0$ of corank~1 consists of all $q$ such that $q(v_0,v_0) = 0$, where $v_0$ is a generator of~$\Ker q_0$.
Thus, the singular points of $D$ of the first type appear if and only if $\P(W)$ contains quadrics of corank~2, and those of the second type appear if and only if $v_0$ is isotropic for 
all quadrics in $\P(W)$ which means that the family of quadrics is not regular.
% Since $\Sing(\Delta)$ is the locus of quadrics of corank $\ge 2$, it follows that $D$ is singular if there are quadrics of corank $\ge 2$ among $Q_w$.
% Furthermore, the tangent space to $D$ at a point $w_0$ such that $q(w_0)$ has corank 1 is the space of all $w \in W$ such that $q(w)(v_0,v_0) = 0$, where $v_0$ is a generator of~$\Ker q(w_0)$.
% Thus $w_0$ is a singular point of $D$ if and only if all quadrics $Q_w$ contain $v_0$, i.e., if $v_0 \in X$, which violates regularity.
\end{proof}

% We denote by
% \begin{equation*}
% Q \to \P(W)
% \end{equation*}
% the family of quadrics passing through $X$.
% In other words, $Q \subset \P(V) \times \P(W)$ is a divisor of bidegree~$(2,1)$ with equation $q \in S^2V^\vee \otimes W^\vee$.
% By definition of $Q \to \P(W)$, it is a flat family of $n$-dimensional quadrics corresponding to $\cL = \OO_{\P(W)}(-1)$ and $\cE = V \otimes \OO_{\P(W)}$ in
% the notation of Section \ref{sec:quadrics}.

It is easy to compute the class of $Q$ in the Grothendieck ring of varieties.

\begin{lemma}\label{lemma:q-x}
The class of $Q$ in the Grothendieck ring can be written as
\begin{equation*}
[Q] = [\P^{n+1}][\P^{m-1}] + [X]\L^m.
\end{equation*}
\end{lemma}
\begin{proof}
The restriction of the map $Q \to \P(V) \cong \P^{n+1}$ to $\P(V) \setminus X$ is a Zariski locally trivial fibration with fiber~$\P^{m-1}$, and its restriction to $X$ is just $X \times \P(W) {} \cong X \times \P^m$.
By Lemma~\ref{lemma:zltf} we have
\begin{equation*}
[Q] = ([\P^{n+1}] - [X])[\P^{m-1}] + [X][\P^m] = [\P^{n+1}][\P^{m-1}] + [X]([\P^m]-[\P^{m-1}]),
\end{equation*}
which is precisely what we need.
\end{proof}

\subsection{Hyperbolic reduction for constant sections}

The above lemma relates the class of the intersection of quadrics $X$ to the class of the family $Q$.
On the other hand, we can use the results of Section~\ref{sec:quadrics}
% structure of a family of quadrics $Q \to \P(W)$ 
to express the class of $Q$ differently.

\begin{lemma}\label{lemma:constant-reduction}
Assume $X$ is smooth and contains a linear space $P \subset X$ of dimension $k$.
Then
\begin{equation*}
[Q] = [\P^m][\P^k](1 + \L^{n-k}) + [\bar{Q}_P]\L^{k+1},
\end{equation*}
where $\bar{Q}_P \to \P(W)$ is a family of quadrics of dimension $n -2k - 2$.
\end{lemma}
\begin{proof}
By Lemma~\ref{lemma:smoothness-criterion} the space $P$ does not intersect the singular locus of any quadric in the family.
Therefore, the section of $F_k(Q/\P(W))$ given by the family $P \times \P(W) \subset Q$ is nondegenerate.
So, Corollary~\ref{proposition:hyperbolic-reduction} applies and gives the required formula.
\end{proof}

The following geometric interpretation of the hyperbolic reduction $\bar{Q}_P$ is useful.
Assume $P = \P(U)$, where $U \subset V$ is a linear subspace, $\dim(U) = k+1$.

\begin{lemma}\label{lemma:bar-q-explicit}
The scheme $\bar{Q}_P$ can be realized as a complete intersection in $\P(W) \times \P(V/U) \cong \P^m \times \P^{n-k}$ 
of $k+1$ divisors of bidegree $(1,1)$ and one divisor of bidegree~$(1,2)$.
Furthermore, all the fibers of the projection $\bar{Q}_P \to \P(V/U) \cong \P^{n-k}$ are linear subspaces in $\P(W) \cong \P^m$, and if $m \ge k + 1$ the general fiber is isomorphic to~$\P^{m-k-2}$
\textup{(}and is empty if $m = k+1$\textup{)}, and the locus over which the fibers jump is the image of the linear projection of $X$ from $\P(U)$ to $\P(V/U)$.
\end{lemma}
\begin{proof}
We apply the description of $\bar{Q}_P$ from Lemma~\ref{lemma:bar-q-explicit-general}.
In our case $S = \P(W)$, $\cE = V \otimes \cO_{\P(W)}$ and $\cU_{k+1} = U \otimes \cO_{\P(W)}$, thus $\P_S(\cE/\cU_{k+1}) = \P(W) \times \P(V/U)$.

Furthermore, the line bundles $\cL^\vee$ and $\OO(H')$ on $\P(W) \times \P(V/U)$
are isomorphic to $\cO(1,0)$ and $\cO(0,1)$ respectively and the short exact sequence~\eqref{eq:seq2} of Lemma~\ref{lemma:bar-q-explicit-general}
rewrites as
\begin{equation*}
0 \to \cO(1,2) \to \cU_{k+2}^\vee \otimes \OO(1,1) \to U^\vee \otimes \cO(1,1) \to 0 
\end{equation*}
and splits. 
Thus $\cU_{k+2} \cong U \otimes \cO \oplus \cO(0,-1)$ is a pullback from $\P(V/U)$ of the vector bundle $U \otimes \cO \oplus \cO(-1)$,
and~$\bar{Q}_P$ is the intersection of $(k+1)$ divisors of bidegree~$(1,1)$ (indexed by a basis in $U$), and a divisor of bidegree~$(1,2)$.

Since the conditions defining $\bar{Q}_P$ are linear along $\P(W)$, all fibers of the projection $\bar{Q}_P \to \P(V/U)$ are linear spaces.
Explicitly the fibers can be described as follows.
The pushforward of $\cU_{k+2}^\vee \otimes \OO(1,1)$ along the projection to $\P(V/U)$ 
is isomorphic to $W^\vee \otimes (U^\vee \otimes \cO(1) \oplus \cO(2))$
and the section $q'$ defines a linear map
\begin{equation}\label{eq:map-fiber-qp1}
W \otimes \cO_{\P(V/U)} \to U^\vee \otimes \cO(1) \oplus \cO(2)
% \cU_{k+2}^\vee(1).  
\end{equation}
For every point $x = U_{k+2}/U \in \P(V/U)$ evaluating the above map at $x$
gives a linear map
\begin{equation}\label{eq:map-fiber-qp}
W \xrightarrow{\ q'(x)\ } U_{k+2}^\vee \otimes (U_{k+2}/U)^\vee
\end{equation}
and the fiber of $\bar{Q}_P \to \P(V/U)$ over $x$ is the projectivization of the kernel of~\eqref{eq:map-fiber-qp}.
Since the dimension of the source of~\eqref{eq:map-fiber-qp} is $m+1$, and the dimension of the target is~$k + 2 {}\le m + 1$, it remains to show that the map is surjective for general $U_{k+2}$ and to describe the set of all $U_{k+2}$ for which the surjectivity fails.
The dual of~\eqref{eq:map-fiber-qp} can be rewritten as
\begin{equation*}
U_{k+2} \otimes (U_{k+2}/U) \to W^\vee. % \xrightarrow{\ q'(x) \ } 
\end{equation*}
Moreover, for any choice of splitting $U_{k+2} \cong U \oplus \langle v \rangle$ the above map on both summands of the decomposition
$(U \otimes \langle v \rangle) \oplus (\langle v \rangle \otimes \langle v \rangle) \cong U_{k+2} \otimes (U_{k+2}/U)$ is given up to a scalar by evaluation of the quadratic form $q$ (but the scalars are different).

Assume that this map is not injective and let $u \in U_{k+2}$ be an element in the kernel.
If $u \not\in U$, we can choose $v = u$ for the splitting, and it follows from the above description that all quadratic forms $q(w)$ vanish on $u$, hence $u \in X$.
In this case the point $U_{k+2}/U$ of $\P(V/U)$ is the projection of $u$ from $\P(U)$.

% it is clear that $u'$ is contained in all quadrics $Q_w$, hence in $X$, and its image $U_{k+2}/U$ in $\P(V/U)$ is then 
% the projection of~$u' \in X$ to $\P(V/U)$.

On the other hand, if $u \in U$ is in the kernel, then the space $U_{k+2}$ is orthogonal to $u$ with respect to all quadratic forms $q(w)$.
Since $u \in \P(U) \subset X$, it follows that the space $\P(U_{k+2})$ is tangent to $X$ at $u$, hence the point $U_{k+2}/U$ of $\P(V/U)$ is on the projection of the tangent space of $X$ at $u$ from $\P(U)$.

% then the space $\P(U_{k+2})$ is tangent to $X$ at $u'$, hence $U_{k+2}/U$ is in the image of the exceptional divisor of the blowup $\Bl_P(X)$ under the projection.
Both arguments work in the opposite direction too, so it follows that the fibers of the map \hbox{$\bar{Q}_P \to \P(V/U)$} jump precisely over the image of $\Bl_P(X)$ under the linear projection $X \dashrightarrow \P(V/U)$.
It remains to note that $\dim X = n - m < n - k = \dim \P(V/U)$, so that the general fiber of the projection $\bar{Q}_P \to \P(V/U)$ (away of the image of $\Bl_P(X)$) indeed has dimension $m-k-2$.
\end{proof}

In what follows we use these observations and results of Section~\ref{sec:quadrics} to produce some relations in the Grothendieck ring of varieties.
Our ultimate goal is to express the class of a family of quadrics $Q$ associated with a complete intersection $X$ of quadrics in terms of the class $\P(W) = [\P^m]$ 
of the base and the class of the determinant double cover $Y \to \P^m$, and compare it with the relation of Lemma~\ref{lemma:q-x}.

\subsection{Naive examples}

In this section we discuss two examples when the above program can be achieved by a reduction along a constant $k$-section of $Q \to \P(W)$ as in Lemma~\ref{lemma:constant-reduction}.

In the first example we relate the classes of two genus 1 curves.

\begin{corollary}\label{corollary:m1}
Let $X \subset \P^3$ be a smooth complete intersection of two quadrics and $Y \to \P^1$ is the corresponding determinant double cover.
If $X$ has a $\kk$-rational point, then $([X] - [Y])\L = 0$.
% where $Y = \tilde{\P}^1$ is the corresponding determinant double cover of $\P^1$.
\end{corollary}
\begin{proof}
% By Lemma~\ref{lemma:constant-reduction} applied to $n=2$, $m=1$ and $k=0$ 
Applying Lemma~\ref{lemma:constant-reduction} with $P$ being a $\kk$-rational point on $X$ we deduce $[Q] = [\P^1](1 + \L^2) + [Y]\L$ 
(the family $\bar{Q}_P$ of $0$-dimensional quadrics coincides with the determinant double cover $Y$ by Lemma~\ref{lemma:qbar-tildes}).
On the other hand, $[Q] = [\P^3] + [X]\L$ by Lemma~\ref{lemma:q-x}.
Since $[\P^1](1 + \L^2) = [\P^3]$, we deduce the desired equality $[X]\L = [Y]\L$.
\end{proof}

% Analogously we can consider the case $m = 2$.

In the second example we relate the classes of two K3 surfaces.

\begin{corollary}\label{corollary:m2-stupid}
Let $X \subset \P^5$ be a smooth complete intersection of three quadrics and $Y \to \P^2$ is the corresponding determinant double cover.
If $X$ has a line defined over $\kk$, then $([X] - [Y])\L^2 = 0$. 
% where $Y = \tilde{\P}^2$ is the corresponding determinant double cover of $\P^2$.
\end{corollary}
\begin{proof}
% By Lemma~\ref{lemma:constant-reduction} applied to $n=4$, $m=2$ and $k=1$ 
Applying Lemma~\ref{lemma:constant-reduction}
with $P$ being a $\kk$-rational line on $X$ we have $[Q] = [\P^2][\P^1](1 + \L^3) + [Y]\L^2$ 
(the family $\bar{Q}_P$ of $0$-dimensional quadrics coincides with the determinant double cover $Y$ by Lemma~\ref{lemma:qbar-tildes}).
On the other hand, $[Q] = [\P^5][\P^1] + [X]\L^2$ by Lemma~\ref{lemma:q-x}.
% Since $[\P^2](1 + \L^3) = [\P^5]$, 
Again we deduce that $[X]\L^2 = [Y]\L^2$.
\end{proof}

However, the two L-equivalences we found here are trivial.

\begin{lemma}
In the situations of Corollary~\textup{\ref{corollary:m1}} and Corollary~\textup{\ref{corollary:m2-stupid}} we have $Y \cong X$. 
In particular, $[X] - [Y] = 0$.
\end{lemma}
\begin{proof}
% Let $k = 0$ in case $m = 1$ and $k = 1$ in case $m = 2$.
In these examples we have $m = k + 1$ and $n = 2m = 2k + 2$.
By Lemma~\ref{lemma:bar-q-explicit} we have a map $Y \to \P^{n-k} = \P^{k+2}$, whose general fiber is empty, and which is birational onto the image of $X$ under the linear projection from $P$.
Thus $X$ and $Y$ are birational. 
But since both are either elliptic curves or K3 surfaces, it follows that $X \cong Y$.
\end{proof}

\subsection{A nontrivial example}

We consider the same example of two K3 surfaces as above, but weaken the assumptions to avoid isomorphism of $X$ and $Y$.

So, let $X \subset \P^5$ be a complete intersection of three quadrics and $Y \to \P^2$ be its determinant double cover.
We assume that both $X$ and $Y$ are smooth.
By Lemma~\ref{lemma:smoothness-criterion} this means that all quadrics in the family $Q \to \P^2$ defining $X$ have corank $\le 1$ and that $X$ does not pass through the singular points of singular quadrics.
% Let $Q \to \P^2$ be the corresponding family of four-dimensional quadrics.

Assume $X$ has a $\kk$-rational point $P$, and let $\bar{Q}_P \to \P^2$ be the family of two-dimensional quadrics obtained by hyperbolic reduction of $Q$.
We want to check whether this family admits a rational section.
Consider the linear projection $\P^5 \dashrightarrow \P^4$ with center in $P$, and let~$X' \subset \P^4$ be the image of $X$.

\begin{lemma}\label{lemma:k3-curves}
Let $X \subset \P^5$ be a smooth complete intersection of three quadrics and $P$ a $\kk$-rational point of $X$ that does not lie on a line in $X$.
Then $X' \cong \Bl_P(X)$ and
\begin{equation}\label{eq:bar-qp-blowup}
\bar{Q}_P \cong \Bl_{X'}(\P^4).
\end{equation}
% where $X'$ is the projection of $X$ from $P$ to $\P^4$.
Moreover, the family of quadrics $\bar{Q}_P \to \P^2$ admits a rational section if and only if there is a curve $C \subset X$ of odd degree defined over $\kk$.
\end{lemma}
\begin{proof}
By Lemma~\ref{lemma:bar-q-explicit} we know that $\bar{Q}_P \subset \P^2 \times \P^4$ is the complete intersection of the divisors of bidegree $(1,1)$ and $(1,2)$.
Applying~\cite[Lemma~2.1]{Kuz16a} to the morphism $\cO_{\P^4}^{\oplus 3} \to \cO_{\P^4}(1) \oplus \cO_{\P^4}(2)$, the incarnation of the morphism~\eqref{eq:map-fiber-qp1} in this particular situation,
we conclude that $\bar{Q}_P$ is isomorphic to the blowup of~$\P^4$ with center in the degeneracy locus of that morphism.
Furthermore, the argument of Lemma~\ref{lemma:bar-q-explicit} identifies this degeneracy locus with $X'$, 
and because $X$ contains no lines passing through $P$ we have
\begin{equation}\label{eq:x-blowup}
X' \cong \Bl_P(X).
\end{equation}
So in the end we deduce~\eqref{eq:bar-qp-blowup}, and prove the first part of the lemma.

Let us prove the second part.
%If $X$ contains a line through $P$, then it gives a section of $\bar{Q}_P \to \P^2$.
By Proposition~\ref{proposition:section-alpha} the map $\bar{Q}_P \to \P^2$ has a rational section if and only if it has a rational multisection of odd degree.
Thus we have to show that there is a cycle $Z \subset \bar{Q}_P$ of codimension 2, 
whose intersection with the class of the fiber of the map $\bar{Q}_P \to \P^2$ is odd if and only if there is a curve of odd degree on $X$.

Because of the blowup representation of $\bar{Q}_P$, its group of codimension 2 cycles is generated by the square of the hyperplane class of $\P^4$, and the group of 1-cycles on $X'$.
More precisely, consider the blowup diagram
\begin{equation*}
\xymatrix{
E \ar[d]_p \ar[r]^i & \bar{Q}_P \ar[d]^\pi \\
X' \ar[r]^j & \P^4
}
\end{equation*}
Then 
\begin{equation*}
\CH^2(\bar{Q}_P) = \pi^*\CH^2(\P^4) \oplus i_*p^*\CH^1(X') = \Z {H'}^2 \oplus i_*p^*\CH^1(X'),
\end{equation*}
% $\CH^2(\bar{Q}_P) = \pi^*\CH^2(\P^4) \oplus i_*p^*\CH^1(X')$, and $\CH^2(\P^4) = \Z H'^2$, 
where $H'$ is the hyperplane class of~$\P^4$.

Let us first relate the pullbacks to $\bar{Q}_P$ of the hyperplane class $h$ of $\P(W)$ to the classes $H'$ and~$E$.
For this we compute the canonical class of $\bar{Q}_P$ in two ways.
The blowup representation implies $K_{\bar{Q}_P} = -5H' + E$, and the complete intersection in $\P^2 \times \P^4$ representation gives 
\begin{equation*}
K_{\bar{Q}_P} = -3h - 5H' + (h + H') + (h + 2H') = -h - 2H',
\end{equation*}
whereof we deduce $h = 3H' - E$.

The class of the fiber of the map $\bar{Q}_P \to \P^2$ is $h^2$.
Thus
we need to compute the parity of the intersection products $h^2\cdot H'^2$ and $h^2\cdot i_*p^*C$ for $C \in \CH^1(X')$.
The first is easy, it can be expressed as an intersection on $\P^2 \times \P^4$
\begin{equation*}
h^2 \cdot H'^2 \cdot (h+H') \cdot (h+2H') = 2h^2H'^4 = 2.
\end{equation*}
Since it is even, we are left with the computation of the second intersection.
This intersection can be rewritten as $p_*i^*h^2 \cdot C$, so we need to compute $p_*i^*(h^2) \in \CH^1(X')$.

We start with $i^*h = 3p^*(H'\vert_{X'}) - i^*(E)$, hence
\begin{equation*}
i^*(h^2) = 9p^*(H'\vert_{X'})^2 - 6p^*(H'\vert_{X'})i^*(E) + i^*(E)^2.
\end{equation*}
Since $p \colon E \to X'$ is a $\P^1$-fibration, and $-i^*(E)$ is its relative hyperplane class,
the pushforward map $p_*\colon \CH^2(E) \to \CH^1(X')$ kills the first summand in the right hand side and takes the second summand to $6H'\vert_{X'}$.
So, it remains to compute $p_*i^*(E)^2$.

For this we use the Grothendieck relation 
\begin{equation*}
i^*(E)^2 - p^*(c_1(\cN))i^*E + p^*(c_2(\cN)) = 0,
\end{equation*}
where $\cN$ is the normal bundle of $X'$ in $\P^4$.
Applying $p_*$ we obtain 
\begin{equation*}
p_*i^*(E)^2 = -c_1(\cN).
\end{equation*}
By adjunction $c_1(\cN) = K_{X'} - K_{\P^4}\vert_{X'} = E' + 5H'\vert_{X'}$, where $E'$ is the exceptional divisor of the blowup $X' \to X$.
So summarizing we deduce
\begin{equation*}
p_* i^*(h^2) = 6H'\vert_{X'} - (E' + 5H'\vert_{X'}) = H'\vert_{X'} - E'.
\end{equation*}
Finally, note that since $X'$ is the projection of $X$ from a point, we have $H'\vert_{X'} = H - E'$, where $H$ is the hyperplane class of $X \subset \P^5$.
Thus $p_*i^*(h^2) = H - 2E'$ and its intersection with $C$ is odd if and only if~$\deg(C) = C \cdot H$ is.
\end{proof}

\begin{corollary}\label{corollary:relation}
Assume $X$ is a smooth complete intersection of three quadrics in $\P^5$ such that the corresponding double cover $Y \to \P^2$ is also smooth.
% such that there are no quadrics of corank~$2$ through $X$, and $Y \to \P^2$ is the corresponding double cover branched in a sextic curve.
If $X$ contains a $\kk$-point not lying on a line and a curve of odd degree then $([X] - [Y])\cdot \L = 0$.
\end{corollary}
\begin{proof}
We express the class of $\bar{Q}_P$ in two ways.
First, we use blowup representations~\eqref{eq:bar-qp-blowup} and~\eqref{eq:x-blowup}:
\begin{equation*}
[\bar{Q}_P] = [\P^4] + [X']\L = [\P^4] + ([X] + \L)\L = [\P^4] + \L^2 + [X]\L.
\end{equation*}
On the other hand, since $X$ contains a curve of odd degree, by Lemma~\ref{lemma:k3-curves} and Theorem~\ref{thm:reduction-dim2} we have
\begin{equation*}
[\bar{Q}_P] = [\P^2](1 + \L^2) + [Y]\L.
\end{equation*}
After cancellation, we get $[X]\L = [Y]\L$.
\end{proof}

\subsection{Proof of Theorem~\textup{\ref{theorem:main}}}

By Proposition~\ref{proposition:section-alpha} and Lemma~\ref{lemma:k3-curves} the Brauer class $\alpha_Y$ vanishes 
if and only if $X$ contains a curve of odd degree, and in this case by Corollary~\ref{corollary:relation} we have $[X]\L = [Y]\L$.
Moreover, by~\cite{Kuz08} we have in this case equivalences $\BD(X) \cong \BD(Y,\alpha_Y) = \BD(Y)$,
and this proves the first part of the Theorem.

Now to prove the second part we discuss, in case~$\kk = \C$, the subset of the moduli space $M$ of polarized K3 surfaces of degree~8 that parametrizes~$X$ such that 
the Brauer class $\alpha_Y$ 
% on the corresponding degree~2 surface $Y$ 
is trivial, but $X \not\cong Y$. 

We denote by $M_d \subset M$ the subset of $M$ parameterizing $X$ with Picard number 2 of discriminant $-d$. 
In other words, we assume that $\operatorname{Pic}(X)$ is generated by the polarization $H$ of degree 8 and a curve class~$C$ such that
\begin{equation*}
d = -\det
\begin{pmatrix}
H^2 & C \cdot H \\ C \cdot H & C^2
\end{pmatrix} =
(C \cdot H)^2 - 8(C^2).
\end{equation*}
Note that $M_d \subset M$ is locally closed of codimension 1.

\begin{lemma}\label{lemma:non-iso}
Assume $\kk = \C$.
There is a countable number of discriminants $d$ such that for any $X \in M_d$ we have $\alpha_Y = 0$ but $X \not\cong Y$.
\end{lemma}
\begin{proof}
By Lemma~\ref{lemma:k3-curves} and Proposition~\ref{proposition:section-alpha} the Brauer class $\alpha_Y$ vanishes if and only if $X$ contains a curve of odd degree, i.e., iff~$C \cdot H$ is odd, i.e., iff~$d \equiv 1 \pmod 8$.
On the other hand,
by~\cite[Theorem~3.1.7]{MN} we have $X \cong Y$ if and only if one of the equations
\begin{equation}\label{eq:arithmetic}
a^2 - db^2 = \pm 8
\end{equation}
has an integer solution.
It is easy to see that when $d$ is an odd square greater than 9, equation~\eqref{eq:arithmetic} does not have integer solutions.
\end{proof}

A simple geometric example of a K3 surface $X$ of degree 8 such that $X \not\cong Y$ and with $\alpha_Y = 0$ is a smooth complete intersection of three quadrics that contains a rational normal cubic curve.
Indeed, then~$C \cdot H = 3$ and $C^2 = -2$, so $d = 25$ and $X \not\cong Y$.

% \begin{proof}[Proof of Theorem~\textup{\ref{theorem:main}}]
% % In Lemma~\ref{lemma:k3-curves} above we have shown that existence of a curve of odd degree on $X$ is equivalent to existence of a rational section
% % of the family $\bar{Q}_P \to \P^2$ of quadrics, which by Proposition~\ref{proposition:section-alpha} is equivalent to vanishing of the Brauer class 
% % $\alpha_Y$ on $Y$ associated with this family of quadrics.
% % 
% The result follows from Corollary~\ref{corollary:relation} and Lemma~\ref{lemma:non-iso}.
% \end{proof}

\providecommand{\arxiv}[1]{{\tt{arXiv:#1}}}

%\medskip
%\medskip
%
%
%\address{
%{\bf Alexander Kuznetsov}\\
%Algebraic Geometry Section, \\
%Steklov Mathematical Institute of Russian Academy of Sciences,\\
%8 Gubkin str., Moscow 119991 Russia.\\
%\\
%The Poncelet Laboratory, \\
%Independent University of Moscow\\
%\\
%Laboratory of Algebraic Geometry, \\
%National Research University Higher School of Economics\\
%\email{akuznet@mi.ras.ru}
%}
%
%\medskip
%\medskip
%
%\address{
%{\bf Evgeny Shinder}\\
%School of Mathematics and Statistics \\
%University of Sheffield \\
%The Hicks Building \\
%Hounsfield Road \\
%Sheffield S3 7RH\\
%\email{eugene.shinder@gmail.com}
%}

\end{document}